\documentclass[a4paper]{article}
\usepackage{amssymb}
\usepackage{amsmath}
\usepackage{amsthm}
\usepackage{hyperref}
\usepackage{enumerate}
\usepackage{mathtools}
\usepackage{url}
\usepackage[T1]{fontenc}

\newtheorem{Theorem} {Theorem} [section]
\newtheorem{Proposition} [Theorem] {Proposition}
\newtheorem{Lemma} [Theorem] {Lemma}
\newtheorem{Corollary} [Theorem] {Corollary}

\newcommand{\Ef}{{\mathbb E}}
\newcommand{\Ff}{{\mathbb F}}

\newcommand{\Rf}{{\mathbb R}}

\newcommand{\cC}{{\mathcal C}}
\newcommand{\cD}{{\mathcal D}}

\newcommand{\cI}{{\mathcal I}}

\newcommand{\PG}{\mathrm{PG}}
\newcommand{\Var}{\mathrm{Var}}

\newcommand{\tr}{\mathrm{tr}}
\newcommand{\eps}{\varepsilon}
\newcommand{\<}{\langle}
\renewcommand{\>}{\rangle} 
\renewcommand{\phi}{\varphi} 
\newcommand{\gauss}[2]{\genfrac{[}{]}{0pt}{}{#1}{#2}}
\newcommand{\coloneq}{\vcentcolon=}      
\newcommand{\eqcolon}{=\vcentcolon}      

\newcommand{\ssfrac}{\genfrac{}{}{}3}

\title{Approximately Strongly Regular Graphs}
\author{
Ferdinand Ihringer
}
\date{09 August 2022}

\begin{document}
\maketitle

\begin{abstract}
  We give variants of the Krein bound and the absolute
  bound for graphs with a spectrum similar to 
  that of a strongly regular graph.
  In particular, we investigate what we call
  approximately strongly regular graphs.
  
  \smallskip 
  
  We apply our results to extremal problems.
  Among other things, we show the following:
  
  \smallskip
  
  (1) Caps in $\PG(n, q)$ for which the number of secants on
  exterior points does not vary too much,
  have size at most $O(q^{\frac34 n})$
  (as $q \rightarrow \infty$ or as $n \rightarrow \infty$).
  
  \smallskip
  
  (2) Optimally pseudorandom $K_m$-free graphs
  of order $v$ and degree $k$ for which the induced subgraph 
  on the common neighborhood of a clique of size $i \leq m-3$ 
  is similar to a strongly regular graph,
  have $k = O(v^{1 - \frac{1}{3m-2i-5}})$.
\end{abstract}


\section{Introduction}

We investigate graphs and families of graphs
which asymptotically behave like strongly regular graphs (SRGs).
In particular, we generalize existence conditions.
Our interest stems from the fact that for some extremal
problems such as the cap set problem or optimally pseudorandom
clique-free graphs (see \S\ref{sec:app}) it is natural 
to look for constructions which behave
very similarly to strongly regular graphs.

\medskip 

All graphs in this document are finite and simple.
Let us repeat some basic facts about strongly 
regular graphs and bounds on their parameters:
Our notation for strongly regular graphs is standard,
cf. \cite{BCN,BvM}.
A strongly regular graph $\Gamma$ with parameters $(v, k, \lambda, \mu)$ 
is a $k$-regular graph (not complete, not edgeless) of order $v$ 
such that two distinct adjacent vertices
have precisely $\lambda$ common neighbors, while two distinct nonadjacent
vertices have precisely $\mu$ common neighbors.
One of the parameters depends on the others: For a fixed vertex $a$,
counting the pairs $(b, c)$ with $a \sim b \sim c \not\sim a$ in two ways
shows that $(v-k-1)\mu = k(k-\lambda-1)$.

Call an eigenvalue of the adjacency matrix of a 
regular graph {\it restricted} if
it has an eigenvector orthogonal to the all-ones vector.
Then, alternatively, a strongly regular graph can be defined as a $k$-regular
graph whose adjacency matrix $A$ has exactly two 
restricted eigenvalues $r \geq 0$ and $s < 0$.
Denote the multiplicity of $r$ by $f$ and 
the multiplicity of $s$ by $g$.
We have the identities
\begin{align*}
    & \lambda - \mu = r+s, && k-\mu = -rs.
\end{align*}
Explicit formulas for $f$ and $g$ can be found using $1+f+g = v$
and $k+fr+gs = \tr(A) = 0$.

{\footnotesize 
\medskip 
As a toy example for this introduction, 
we consider the parameter set 
$v = (1+o(1))\lambda^{11}$, $k = (1+o(1))\lambda^{10}$, and $\mu = (1+o(1)) \lambda^9$
(as $\lambda \rightarrow \infty$). 
See \S\ref{sec:bigO} for a discussion of big-$O$ (and similar) notation.\par}

\medskip 

The Krein bound\footnote{Named somewhat indirectly after 
Mark Grigorievich Krein, cf. \cite[p. 26]{BvM}.} and 
the absolute bound provide asymptotic conditions on 
the parameters $(v, k, \lambda, \mu)$ of a strongly regular graph.

\begin{Theorem}[Krein Bound for SRGs, {\cite[p. 26]{BvM}}]\label{thm:krein_srg}
    The eigenvalues $k \geq r \geq 0 > s$ of a strongly regular graph satisfy
    \begin{align*}
        &1 + \frac{s^3}{k^2} - \frac{(s+1)^3}{(v-k-1)^2} \geq 0, &&1 + \frac{r^3}{k^2} - \frac{(r+1)^3}{(v-k-1)^2} \geq 0.
    \end{align*}
\end{Theorem}

{\medskip \footnotesize 
In the toy example above, $s = (-1+o(1)) \lambda^9$, so Theorem \ref{thm:krein_srg}
implies $1 + (-1+o(1)) \lambda^7 - (1+o(1)) \lambda^5 \geq 0$ which is impossible.\par}

\smallskip 

The absolute bound for strongly regular graphs is a corollary of the 
well-known result by Delsarte, Goethals, and Seidel  that 
a family of $n$ unit vector in $\Rf^d$ with at most three distinct inner products
satisfies $n \leq \frac12 d(d+3)$, see Theorem 4.8 and Theorem 4.11 in \cite{DGS1977}.

\begin{Theorem}[Absolute Bound for SRGs, {\cite[Prop. 1.3.14]{BvM}}]\label{thm:absolute_srg}
  The multiplicities $f,g$ of a primitive strongly regular graph 
  satisfy $v \leq \frac12 f(f+3)$ and $v \leq \frac12 g(g+3)$.
\end{Theorem}

{\medskip \footnotesize
In the toy example above, $g = (1+o(1)) \lambda^3$, so Theorem \ref{thm:absolute_srg}
implies $(1+o(1)) \lambda^{11} \leq \frac12 \lambda^6$ which is impossible.
\par}

\smallskip 

In the first part of this document,
we generalize Theorem \ref{thm:krein_srg}.

\begin{Proposition}[Krein Bound, Variant for Regular Graphs] \label{prop:krein_regular}
    Let $\Gamma$ be a $k$-regular graph of order $v$ with adjacency matrix $A$.
    Let $r$ denote the second largest and $s$ the smallest eigenvalue of $A$. Then
    \begin{align*}
        (s+r^2)v + 2(k-r)(r-s) \geq 0, && (r+s^2)v + 2(k-s)(s-r) \geq 0.
    \end{align*}
\end{Proposition}

An {\it $1$-walk-regular graph} is a graph
in which the number of walks of length $\ell$ 
between vertices $a$ and $b$ with $a = b$ or $a,b$ adjacent
only depends on $\ell$ and $a = b$, 
not the choice of $a$ and $b$, cf. \cite{DFG2009}.
Arc-transitive graphs and strongly regular 
graphs are examples for $1$-walk-regular graphs.
In \S\ref{sec:bounds}, 
we will provide a variant of Proposition \ref{prop:krein_regular}
for a special type of 1-walk-regular graphs 
which is significantly stronger.

\smallskip 

There is a poor man's version of the absolute bound which only shows $v \leq f^2$
and $v \leq g^2$. We give a variant of this poor man's result.

\begin{Proposition}[Absolute Bound, Variant]\label{prop:abs_reg}
    Consider a $k$-regular graph of order $v$ with adjacency matrix $A$.
    Let $r,s$ be real numbers with $k > r \geq 0 > s$.
    Suppose that $A$ has at least $f_1$ and 
    at most $f_2$ restricted eigenvalues in $[r, k]$,
    all eigenvalues of $A$ are at least $s-\eps$ for some $\eps > 0$,
    and at least $v-f_2$ eigenvalues of $A$ are in $[s-\eps,s]$.
    If $s^2+s > \eps$,
    then $v \leq f_2(f_2+1)-f_1$.
\end{Proposition}

In the second part of this document, we consider what 
we call approximately strongly regular graphs.
For two adjacent vertices $a$ and $b$ of a graph $\Gamma$, let $\lambda_{ab}$ denote
the number of common neighbors of $a$ and $b$ in $\Gamma$. 
Similarly, for two distinct nonadjacent vertices $a$ and $b$ of a graph $\Gamma$,
let $\mu_{ab}$ denote the number of common neighbors of $a$ and $b$ in $\Gamma$.
Let $\Lambda$ (respectively, $M$) denote the set of all pairs of adjacent
(respectively, distinct nonadjacent) vertices in $\Gamma$.


We call a $k$-regular graph (not complete, not edgeless) 
$\Gamma$ of order $v$ an {\it approximately strongly regular graph}
with parameters $(v, k, \lambda, \mu; \sigma)$, where $\sigma \geq 0$,
if $\Ef(\lambda_{ab}) \coloneq \frac{1}{|\Lambda|} \sum_{(a, b) \in \Lambda} \lambda_{ab} = \lambda$ 
and $\Var(\lambda_{ab}) \coloneq \frac{1}{|\Lambda|} \sum_{(a, b) \in \Lambda} (\lambda_{ab}-\lambda)^2 \leq \sigma^2$, 
and $\Ef(\mu_{ab}) \coloneq \frac{1}{|M|} \sum_{(a,b) \in M} \mu_{ab} = \mu$
and $\Var(\mu_{ab}) \coloneq \frac{1}{|M|} \sum_{(a,b) \in M} (\mu_{ab}-\mu)^2 \leq \sigma^2$.

Strongly regular graphs are precisely the approximately strongly
regular graphs with $\sigma=0$.
The complement of an approximately strongly regular graph
with parameters $(v, k, \lambda, \mu; \sigma)$ is an approximately
strongly regular graph with parameters $(v, v-k-1, v-2k+\mu, v-2k+\lambda; \sigma)$.
Counting triples $(a, b, c)$ with $a \sim b \sim c \not\sim a$
shows $\sum_{(a,c) \in M} \mu_{ac} = \sum_{(a,b) \in \Lambda} (k-\lambda_{ab}-1)$.
Hence, $(v-k-1)\mu = k(k-\lambda-1)$ also holds for 
approximately strongly regular graphs.

\smallskip 

{\medskip \footnotesize 
In our toy example with $v = (1+o(1))\lambda^{11}$, $k = (1+o(1))\lambda^{10}$, and $\mu = (1+o(1)) \lambda^9$, 
Proposition \ref{prop:krein_regular} rules out the existence of approximately strongly regular graphs
with $\sigma = o(\lambda^{2.5})$.
Under slightly stronger conditions, see Proposition \ref{prop:krein_1walk}, 
we also obtain $\sigma = o(\lambda^{8})$.
If our toy example contains a coclique of size $(1+o(1)) q^9$,
then we will also rule out $\sigma = o(\lambda^{3.5})$.\par}

\medskip 

In the third part of this document, we apply our results to the cap set problem
and to optimally pseudorandom clique-free graphs.

{\medskip \footnotesize 
For instance, if there exists a cap of size $(1+o(1))q^9$ in 
the projective space $\PG(10, q)$, then 
a standard construction yields a approximately strongly regular graph with 
the same parameters of our toy example (where $\lambda = q-2$).\par} 

\section{Bounds}\label{sec:bounds}

Denote the all-ones vector by $j$, the all-ones matrix by $J$, 
and the identity matrix by $I$. We denote the Hadamard product
of two matrices by $\circ$.

\subsection{Krein Bounds} \label{sec:krein}

{\smallskip\footnotesize
The following proof is based on Remark (i) on page 50 in \cite{BCN}. \par}

\begin{proof}[Proof of Proposition \ref{prop:krein_regular}]
  Consider the matrices $E_1$ and $E_2$ defined by
    \begin{align*}
        & E_1 = \frac{1}{r-s} \left( A - sI - \tfrac{k-s}{v} J\right),
        && E_2 = \frac{1}{s-r} \left( A - rI - \tfrac{k-r}{v} J\right).
    \end{align*}
    The spectrum of $E_2$ is in $[0, 1]$
    as $(s-r)E_2$ has only eigenvalues in $[s, r]$.
    Write $E_2 \circ E_2$ as a linear combination of the matrices $A, I, J$.
    Then the coefficients of $I$ and $A$ are
    \begin{align*}
        &  \frac{1}{(s-r)^2} \left( r^2 + 2 r \tfrac{k-r}{v} \right) \text{ for } I,
        &&  \frac{1}{(s-r)^2} \left( 1 - 2 \tfrac{k-r}{v} \right) && \text{ for } A.
    \end{align*}
    Now we write $E_2 \circ E_2$ as a linear combination of the matrices $E_1, E_2, J$,
    that is we replace $A$ and $I$ by $E_1$ and $E_2$. We obtain the coefficients
    \begin{align*}
        &\frac{r + r^2}{(s-r)^2} \text{ for $E_1$, }
        &&
        t \coloneq \frac{(s + r^2) v + 2(k-r)(r-s)}{v (s-r)^2} \text{ for $E_2$}.
    \end{align*}
    Let $\chi$ be an eigenvector of $A$ with $A\chi = s\chi$. 
    Then $E_1 \chi = 0$. Hence, $(E_2 \circ E_2) \chi =  t E_2 \chi = t \chi$.
    Hence, $\chi$ is an eigenvector
    of $E_2 \circ E_2$ which shows that $t \geq 0$.
    Hence, using our expression for $\beta$ and $u_i - r < u_i$, 
    we obtain the first inequality.
    
    \smallskip 
    For the second inequality, consider $E_1 \circ E_1$ instead of $E_2 \circ E_2$.
\end{proof}

For an eigenspace $U_i$ of a real symmetric matrix $A$ with eigenvalue $u_i$, 
let $F_i$ be the orthogonal projection onto $U_i$, so $F_i$ is idempotent
and $AF_i = u_iF_i$.
We will use repeatedly without further notice
that the eigenspaces of $A$ are pairwise orthogonal,
so $F_iF_j = 0$ if $u_i \neq u_j$.
For the remainder of this subsection,
consider the case that $\Gamma$ is $1$-walk-regular.
Recall that $(A^\ell)_{ab}$ is the number of walks
from $a$ to $b$ of length $\ell$.
It follows from $A^\ell = \sum_i u_i^\ell F_i$ that
$(F_i)_{aa}$ is constant for all vertices $a$,
and that $(F_i)_{ab}$ is constant for all $a,b$ adjacent,
cf. \cite[Theorem 3.1]{DFG2009}.

Let $m_i$ denote the mutiplicity of $F_i$.
From $m_i = \tr(F_i)$ we obtain that $(F_i)_{aa} = \tfrac{m_i}{v}$.
We have $u_i \cdot \frac{m_i}{v} = (u_i F_i)_{aa} = 
(AF_i)_{aa} = k (F_i)_{ab}$ for $a,b$ adjacent,
  so 
  \[
    (F_i)_{ab} = \tfrac{m_i u_i}{vk}.
  \]
Hence, for a $1$-walk-regular graph 
we can control $I \circ F_i$ and 
$A \circ F_i$ as we could 
control $I \circ J$ and $A \circ J$
in the proof of Proposition \ref{prop:krein_regular}.

\medskip 

For a matrix $M$, let $\rho(M)$ denote its spectral radius.
Now we are ready to give an example for how one can increase
regularity conditions in Krein-type bounds to obtain better 
nonexistence results.

\begin{Proposition}[Krein Bound, Variant for 1-Walk-Regular Graphs]\label{prop:krein_1walk}
  Let $\Gamma$ be a $k$-regular $1$-walk-regular graph of order $v$ with adjacency matrix $A$.
  Let $s$ be the smallest eigenvalue of $A$.
  For some $r \geq 0$, let $\cI$ denote the set 
  of indices of eigenvalues $u_i$ of $A$ 
  in the interval $(r, k]$.
  Put $K_1 = \frac{2}{vk} \sum_{i \in \cI} u_i^2$, 
  $K_2 = \frac{2}{v} \sum_{i \in \cI} u$,
  $L = \rho(\sum_{i,j \in \cI} |u_iu_j| F_i \circ F_j)$.
  Then
  \begin{align*}
   (1-K_1) s + r^2 + r K_2 + L \geq 0.
  \end{align*}
\end{Proposition}
\begin{proof}
  Consider the matrices $E_1$ and $E_2$ defined by
    \begin{align*}
        & E_1 = \frac{1}{r-s} \left( A - sI - \sum_{i \in \cI} (u_i - s) F_i\right),\\
        & E_2 = \frac{1}{s-r} \left( A - rI - \sum_{i \in \cI} (u_i - r) F_i\right).
    \end{align*}
    The spectrum of $E_2$ is in $[0, 1]$
    as $(s-r)E_2$ has only eigenvalues in $[s, r]$.
    Write 
    \[ E_2 \circ E_2 
    - \sum_{i, j \in \cI} \frac{(u_i-r)(u_j-r)}{(s-r)^2} F_i \circ F_j \]
    as a linear combination of the matrices $I$ and $A$.
    Put $\alpha_i = (F_i)_{aa}$ and $\beta_i = (F_i)_{ab}$
    for adjacent vertices $a,b$.
    Then the coefficients of $I$ and $A$ are
    \begin{align*}
        &  \frac{1}{(s-r)^2} \left( r^2 
          + 2r \sum_{i \in \cI} (u_i - r)\alpha_i \right) && \text{ for } I,\\
        &  \frac{1}{(s-r)^2} \left( 1 
          - 2\sum_{i \in \cI} (u_i - r)\beta_i \right) && \text{ for } A.
    \end{align*}
    Put $\alpha = 2r \sum_i (u_i - r)\alpha_i$ and $\beta = 2\sum_i (u_i - r)\beta_i$.
    Put
    \begin{align*}
     & \tilde{E}_1 = \frac{1}{r-s} \left( A - sI \right), 
     && \tilde{E}_2 = \frac{1}{s-r} \left( A - rI \right).
    \end{align*}
    Now we write $E_2 \circ E_2$ as a linear combination of the matrices $\tilde{E}_1, \tilde{E}_2$,
    that is we replace $I$ and $A$ by $\tilde{E}_1$ and $\tilde{E}_2$. We obtain the coefficients
    \begin{align*}
        &\frac{(1-\beta)r + r^2 + \alpha}{(s-r)^2} \text{ for $\tilde{E}_1$, }
        &&
        \frac{(1-\beta)s + r^2 + \alpha}{(s-r)^2} \eqcolon t \text{ for $\tilde{E}_2$}.
    \end{align*}
    Let $\chi$ be an eigenvector of $A$ with $A\chi = s\chi$ and $\| \chi \| = 1$. 
    Then $\tilde{E}_1 \chi = 0$. Hence, 
    \begin{align*}
     (E_2 \circ E_2) \chi
     &=  t \tilde{E}_2 \chi 
        + \frac{1}{(s-r)^2} \sum_{i,j \in \cI} |u_iu_j| 
            \cdot (F_i \circ F_j) \chi.
    \end{align*}
    We have $t \tilde{E}_2 \chi = t \chi$ and 
    \[
      \left\| \frac{1}{(s-r)^2} \sum_{i,j \in \cI} |u_iu_j| 
            \cdot (F_i \circ F_j) \chi \right\| \leq \frac{L}{(s-r)^2}.
    \]
    The matrix $E_2 \circ E_2$ is a principle submatrix 
    of $E_2 \otimes E_2$, so the eigenvalues of $E_2 \circ E_2$ interlace
    those of $E_2 \otimes E_2$, thus they are in $[0, 1]$.
    In particular, $E_2 \circ E_2$ is positive semidefinite.
    Hence, $t + \frac{L}{(s-r)^2} \geq 0$.
    Hence, using our expression for $\beta$ and $u_i - r < u_i$, 
    we obtain the assertion.
\end{proof}

We still lack control over $J \circ F_i$
and, more generally, $F_i \circ F_j$. 
For this, we need one last 
concept: We say that an eigenvalue $u_i$ 
of a $1$-walk-regular is {\it $h$-flat}
if for $a,c$ nonadjacent, we have 
\[
  |(F_i)_{ac}| \leq h \tfrac{m_i u_i}{vk}.
\]
Recall that the spectral radius $\rho(M)$ of a matrix $M$
is at most its $1$-norm $\| M\|_1$. Hence,
\begin{align*}
  &\rho(J \circ F_i) \leq \frac{m_i}{v} + k \cdot \tfrac{m_i |u_i|}{vk}
  + (v-k-1) \cdot h \tfrac{m_i |u_i|}{vk}, ~\text{and, more generally,} \\
  &\rho(F_i \circ F_j) \leq 
  \frac{m_i m_j}{v^2} + k \cdot \tfrac{m_i m_j |u_i u_j|}{v^2k^2}
  + (v-k-1) \cdot h^2 \tfrac{m_i m_j |u_i u_j|}{v^2k^2}.
\end{align*}
For a set $\cI$, we need a convenient bound on 
$\sum_{i,j \in \cI} \rho( u_i u_j F_i \circ F_j)$.
Put $M = \sum_{i \in \cI} m_i$.
As an example, suppose that the $u_i$ are $1$-flat for $i \in \cI$.
Then 
\begin{align}
    &\rho\left( \sum_{i,j \in \cI} |u_i u_j| \cdot F_i \circ F_j  \right)
    \leq \sum_{i,j \in \cI} |u_iu_j| \rho(F_i \circ F_j) \notag \\
    & \leq \sum_{i,j \in \cI} \left( \frac{u_i^2u_j^2 m_i m_j}{vk^2} 
        + \frac{|u_iu_j| \cdot m_i m_j}{v^2} \right) \notag \\
    & \leq \frac{1}{vk^2} \sum_{i \in \cI} u_i^2 m_i \sum_{j \in \cI} u_j^2 m_j
            + \sum_{i,j \in \cI} |u_i u_j| \frac{m_im_j}{v^2}  \label{eq:bnd_FiFj} \\
    & \leq \frac{1}{vk^2} \left( \sum_{i \in \cI} m_i u_i^2 \right)^2 
            + \sum_{i \in \cI} 2 u_i^2 \frac{m_i}{v} \notag \\
    & \leq \frac{1}{vk^2} \left( \sum_{i \in \cI} m_i u_i^2 \right)^2 
            + \frac{2}{v} \left( \sum_{i \in \cI} m_i u_i^2 \right).
    \notag
\end{align}

Assuming that the matrices are $1$-flat might be very generous
for cases where $v$ is much larger than $k$.
As $F_i^2 = F_i$, we have 
\[
 \frac{m_i^2}{v^2} + k \frac{m_i^2 u_i^2}{v^2} + \sum_{a \not\sim c} (F_i)_{ac}^2
 = (F_i)_{aa}^2 + \sum_{b \sim c} (F_i)_{ab}^2 + \sum_{a \not\sim c} (F_i)_{ac}^2 = (F_i)_{aa} = \frac{m_i}{k_i}.
\]
Hence, 
\begin{align}
 \Ef((F_i)_{ac}^2) = \frac{\frac{m_i}{k_i} - \frac{m_i^2}{v^2} - k \frac{m_i^2 u_i^2}{v^2}}{v-k-1}.\label{eq:exp_flat}
\end{align}
Thus, if $k$ is small compared to $v$,
then we expect $(F_i)_{ac}$ for $a,c$ nonadjacent
to be small compared to $(F_i)_{ab}$ for $a,b$ adjacent.
Hence, the condition of being $1$-flat in \eqref{eq:bnd_FiFj}
is not particularly strong when $k$ is small compared to $v$.
In particular, later we will consider applications
with $k=o(v)$.

\subsection{Absolute Bounds}

The following is based on the proof of Theorem 2.3.3 in \cite{BCN}.
We use eigenvalue interlacing, cf. \cite{Haemers1995}.
More precisely, if $A$ is a real symmetric matrix of order $v$ with eigenvalues 
$u_1 \geq \ldots \geq u_v$ and $B$ is a principal submatrix of $A$
of order $w$ with eigenvalues $\nu_1 \geq \ldots \geq \nu_w$,
then $u_i \geq \nu_i \geq u_{i-w+v}$.

\begin{proof}[Proof of Proposition \ref{prop:abs_reg}]
    Consider $M \coloneq A - sI - \frac{k-s}{v} J$.
    Then $M$ has at least $v - f_2$ eigenvalues in $[-\eps, 0]$,
    and between $f_1$ and $f_2$ eigenvalues at least $r-s$.
    Hence, $M \otimes M$ has at most $f_2^2$ eigenvalues at least $(r-s)^2$,
    while all its other eigenvalues are at most $\eps^2$.
    Furthermore,
    \[
        M \circ M = (1 - \tfrac{k-s}{v}) A + 
        \left( s^2 + s \tfrac{k-s}{v} \right) I + \left( \tfrac{k-s}{v} \right)^2 J.
    \]
    Hence, $M \circ M$ has one eigenvalue 
    $(k - k \frac{k-s}{v}) + (s^2 + s \frac{k-s}{v}) + 
    \left( \frac{k-s}{v} \right)^2 v \geq s^2+s$ with eigenvector $j$,
    at least $f_1$ eigenvalues 
    at least $(r - r \frac{k-s}{v}) + (s^2 + s\frac{k-s}{v}) \geq s^2+s$ and 
    at least $v-f_2$ eigenvalues 
    at least $(s - s \frac{k-s}{v}) + (s^2 + s \frac{k-s}{v}) = s^2+s$.
    The matrix $M \circ M$ is a principal submatrix of $M \otimes M$,
    so the eigenvalues of $M \circ M$ interlace those of $M \otimes M$.
    Hence, $M \circ M$ has at most $f_2^2$ eigenvalues greater than $\eps^2 < s^2+s$.
    We obtain that $v-f_2+f_1 \leq f_2^2$.
\end{proof}

\subsection{Cvetkovi\'{c} Bound or Inertia Bound}\label{sec:inertia}

The following bound will prove useful for some of our applications.
For a graph $\Gamma$ of order $v$,
let $M$ be a matrix with $M_{ac} = 0$ if $a,c$ nonadjacent.
Let $n^+(M)$ denote the number of positive eigenvalues of $M$
and let $n^-(M)$ denote the number of negative eigenvalues of $M$.
Then a coclique (independent set, stable set) of $\Gamma$ has size at most 
\[
    \min( v - n^+(M), v - n^-(M) ).
\]
This bound is known as  {\it Cvetkovi\'{c} bound} or {\it inertia bound},
cf. \cite[p. 13]{BvM}.
Here we always use the adjacency matrix for $M$.

\section{Approximately Strongly Regular Graphs}

Consider the adjacency matrix $A$ of an approximately strongly regular graph $\Gamma$
with parameters $(v, k, \lambda, \mu; \sigma)$.
We can write $A^2 = kI + \lambda A + \mu (J - I - A) + E$, where
$(E)_{ab} = \lambda_{ab} - \lambda$ when $a,b$ are adjacent,
$(E)_{ab} = \mu_{ab} - \mu$ when $a,b$ are distinct and nonadjacent,
and $(E)_{ab} = 0$ when $a=b$.

Let $\chi$ be an eigenvector of $A$ orthogonal 
to the all-ones vector $j$ with eigenvalue $u$.
Then $u^2 \chi = A^2 \chi = (k-\mu) \chi + (\lambda-\mu) u \chi + E \chi$.
Hence, $\chi$ is an eigenvector of $E$ with some eigenvalue $\nu$.
By solving for $u$, we find that
\[
    u = \frac12 \left( (\lambda-\mu) \pm \sqrt{ (\lambda-\mu)^2 + 4(k-\mu+\nu) } \right).
\]
We say that $u$ has {\it positive form} if
\[
    u = \frac12 \left( (\lambda-\mu) + \sqrt{ (\lambda-\mu)^2 + 4(k-\mu+\nu) } \right),
\]
and that $u$ has {\it negative form} if
\[
    u = \frac12 \left( (\lambda-\mu) - \sqrt{ (\lambda-\mu)^2 + 4(k-\mu+\nu) } \right).
\]
Let $u_1, u_2, \ldots, u_v$ denote the eigenvalues of $A$.
For $u_i$ an eigenvalue of $A$, let $\nu_i$ denote the corresponding eigenvalue of $E$.

\bigskip 

The next result shows that if $\sigma$ and $v$ are sufficiently small,
then there are few (if any) large $\nu_i$.

\begin{Lemma}\label{lem:evs_E}
    The eigenvalues $\nu_1, \ldots, \nu_v$ of $E$ satisfy 
    $\sum \nu_i^2 \leq v(v-1) \sigma^2$.
\end{Lemma}
\begin{proof}
    We have
    \[ 
        \sum \nu_i^2 = \tr(E^2) = \sum_{a \sim b} (\lambda_{ab}-\lambda)^2 
        + \sum_{a \not\sim b} (\mu_{ab}-\mu)^2 \leq v(v-1) \sigma^2. \qedhere
    \]
\end{proof}

Call a graph $\Gamma$ {\it edge-regular} if $\Var(\lambda_{ab}) = 0$
and {\it coedge-regular} if $\Var(\mu_{ab}) = 0$.
Clearly, Lemma \ref{lem:evs_E} can be improved to $\sum \nu_i^2 \leq vk \sigma^2$
for edge-regular graphs and to $\sum \nu_i^2 \leq v(v-k-1) \sigma^2$
for coedge-regular graphs.

\medskip 

In the introduction, we define strongly regular graphs 
in two ways, combinatorially and spectrally.
Lemma \ref{lem:evs_E} shows that that for small $\sigma$,
the restricted eigenvalues of $\Gamma$ are 
concentrated at two values.
Let us also show the reverse,
namely if $\Gamma$ restricted eigenvalues
are concentrated around two values,
then $\Gamma$ is approximately strongly regular.

\smallskip 

Call a $k$-regular (not complete, not edgeless) 
graph $\Gamma$ of order $v$ {\it spectrally approximately strongly regular}
with parameters $(v, k, r, s; \sigma)$ if 
\[
    (A - rI)(A-sI) = \mu J + \tilde{E},
\]
for some constant $\mu$, where the eigenvalues $\tilde{\nu}$ of $\tilde{E}$
are also eigenvalues of $A$ and satisfy $\sum \tilde{\nu}^2 \leq v(v-1) \sigma^2$.
Furthermore,
\[
    A^2 = (r+s-\mu) A + (\mu-rs) I + \mu J + \tilde{E}.
\]
We see that $\Gamma$ is spectrally approximately
strongly regular with parameters $(v, k, r, s; \sigma)$
if and only if $\Gamma$ is approximately strongly regular
with parameters $(v, k, r+s-k-rs, k+rs; \sigma)$.

\subsection{Big-\texorpdfstring{$O$}{O} Notation} \label{sec:bigO}

We use the symbols $O, \Omega, \Theta, o, \omega$ in the following way:
\begin{align*}
    & f(x) = O(g(x)) \text{ (as $x \rightarrow a$)} && \text{ if and only if } 
            &&\limsup_{x \rightarrow a} \tfrac{|f(x)|}{g(x)} < \infty,\\
    & f(x) = \Omega(g(x)) \text{ (as $x \rightarrow a$)} && \text{ if and only if } 
            &&\limsup_{x \rightarrow a} \tfrac{|f(x)|}{g(x)} > 0,\\
    & f(x) = \Theta(g(x)) \text{ (as $x \rightarrow a$)} && \text{ if and only if } 
            &&0 < \limsup_{x \rightarrow a} \tfrac{|f(x)|}{g(x)} < \infty,\\
    & f(x) = o(g(x)) \text{ (as $x \rightarrow a$)} && \text{ if and only if } 
            &&\lim_{x \rightarrow a} \tfrac{|f(x)|}{g(x)} = 0,\\
    & f(x) = \omega(g(x)) \text{ (as $x \rightarrow a$)} && \text{ if and only if } 
            &&\lim_{x \rightarrow a} \tfrac{|f(x)|}{g(x)} = \infty.
\end{align*}

Usually, we have $a = \infty$.
If there are several variables involved, then 
we specify the relevant one.
We also use the big-$O$ notation in minor order terms.
For instance, we can write $x^2+x+100 = x^2 + O(x)$ (as $x \rightarrow \infty$)
as $x+100 = O(x)$.
For us the sign of $f(x)$ is often important, so for convenience,
we aim to use big-$O$ notation with $f(x) > 0$.
For instance, we write $-s = O(\mu)$ even though $s = O(\mu)$ is equally correct.

\smallskip 

If we talk about a $k$-regular graph $\Gamma$ of order $v$
with $k = O(g(v))$ for some function $g$, then we mean that we consider
an infinite family of graphs $(\Gamma_n)$, where $\Gamma_n$
is of order $v_n$ and $k_n$-regular with $k_n = O(g(v_n))$ as $n \rightarrow \infty$.
In particular, if we say that $\Gamma$ is an approximately strongly regular graph 
or a family of approximately strongly regular graphs
with parameters $(v, k, \lambda, \mu; \sigma)$ and 
$k = o(|\mu-\lambda|^{\frac32})$, then there is an 
infinite family of approximately strongly regular graphs $(\Gamma_n)$
with parameters $(v_n, k_n, \lambda_n, \mu_n; \sigma_n)$
such that $k_n = o(|\mu_n-\lambda_n|^{\frac32})$ as $n \rightarrow \infty$.

We also use big-$O$ notation for the eigenvalues $u_i$ and $\nu_i$:
Assume that $|\nu_1| \geq |\nu_2| \geq \cdots \geq |\nu_v|$.
If we write $\nu_i = O(g(n))$, then there is 
a function $h(n)$ such that $\nu_{h(n)} = O(g(n))$.
For instance, we might assume that $\nu_i = O(\mu_n)$
and show some property (P) for $\nu_i$.
By Lemma \ref{lem:evs_E}, $\sum \nu_i^2 \leq v^2 \sigma^2$,
so the number of $\nu_i$ with $\nu_i = \omega(\mu)$ 
is at most $o(\frac{v\sigma}{\mu})$.
Thus, (P) holds for $\nu_{h(n)}$ 
with $h(n) = \Omega(\frac{v_n\sigma_n}{\mu_n})$.

\subsection{Asymptotic Bounds on Eigenvalues}

\begin{Lemma}\label{lem:approx}
    Let $f: \Rf \rightarrow \Rf$ with $f(y) = o(y^2)$ (as $y \rightarrow \infty$).
    Then 
    \begin{align*}
        \sqrt{y^2 + f(y)} - y = (\tfrac12 + o(1)) \tfrac{f(y)}{y}.
    \end{align*}
\end{Lemma}
\begin{proof}
    Write $\sqrt{y^2 + f(y)} - y = y (\sqrt{1+\frac{f(y)}{y^2}} - 1)$.
    The Taylor expansion of $\sqrt{\cdot}$ at $1$ shows that (as $x \rightarrow 0$)
    \[
        \sqrt{1+x} = 1 + \tfrac12 x + O(x^2). \qedhere
    \]
\end{proof}

The following gives us approximate versions 
of the equations $\mu - \lambda = r+s$ and $k-\mu = -rs$ for strongly regular graphs.

\begin{Lemma}\label{lem:asrg_ev2}
    For a family of approximately strongly regular graphs
    with parameters $(v, k, \lambda, \mu; \sigma)$,
    consider an eigenvalue $u_i$ with associated eigenvalue $\nu_i$. 
    
    \smallskip 
    
    \noindent
    If $\mu > \lambda$, $k = o(|\lambda - \mu|^2)$, 
    and $|\nu_i| = o(|\lambda - \mu|^2)$, then the following holds:
    \begin{enumerate}[(i)]
     \item If $u_i$ has positive form,
            then $u_i = (1+o(1)) \frac{k-\mu+\nu_i}{\mu-\lambda}$.
     \item If $u_i$ has negative form,
            then $u_i = -(1+o(1)) (\mu-\lambda)$.
    \end{enumerate}
    If $\mu < \lambda$, $k = o(|\lambda - \mu|^2)$,
    and $|\nu_i| = o(|\lambda - \mu|^2)$, then the following holds:
    \begin{enumerate}[(i)]
     \item[(iii)] If $u_i$ has positive form,
            then $u_i = (1+o(1)) (\lambda - \mu)$.
     \item[(iv)] If $u_i$ has negative form,
            then $u_i = -(1+o(1)) \frac{k-\mu+\nu_i}{\lambda-\mu}$.
    \end{enumerate}
    If $k = \Omega(|\lambda-\mu|^2)$ and $|\nu_i| = O(k)$, then the following holds:
    \begin{enumerate}[(i)]
     \item[(v)] We have $u_i = \Theta(\sqrt{k})$.
    \end{enumerate}
    If $|\nu_i| = \Omega(|\lambda-\mu|^2)$ and $|\nu_i| = \Omega(k)$, 
    then the following holds:
    \begin{enumerate}[(i)]
     \item[(vi)] We have $u_i = \Theta(\sqrt{|\nu_i|})$.
    \end{enumerate}
\end{Lemma}
\begin{proof}
    We only show (i) and (ii) as the other cases are similar.
    Using Lemma \ref{lem:approx} with $y=\mu-\lambda$ and $\mu = o(k)$, 
    we find that if $u_i$ has positive form, then
    \begin{align*}
        u_i &= \tfrac12 \left( \lambda - \mu + (1+o(1)) \sqrt{(\lambda-\mu)^2 - 4(k-\mu+\nu_i)} \right)\\
        &= (1+o(1)) \tfrac{k-\mu+\nu_i}{\mu-\lambda}.
    \end{align*}
    If $u_i$ has negative form, then
    \begin{align*}
        u_i &= \tfrac12 \left( \lambda - \mu - (1+o(1)) \sqrt{(\lambda-\mu)^2 - 4(k-\mu+\nu_i)} \right)\\
        &= -(1+o(1)) (\mu-\lambda). \qedhere
    \end{align*}
\end{proof}

\subsection{Krein Bounds}

The expected size of the common neighborhood of two distinct
vertices is $(1+o(1)) \frac{k^2}{v}$. Hence, if $k = o(v)$,
then $\mu = (1+o(1)) \frac{k^2}{v}$, so $\mu = o(k)$.

\begin{Proposition}[Krein Bound, Approximately Strongly Regular Graphs]\label{prop:krein_bnd_for_asrg}
    Consider a family of approximately strongly regular graphs with
    $\mu > \lambda$, $k=o(v)$, and $k = o(|\mu-\lambda|^{\frac32})$. 
    Then $\sigma \geq (1+o(1)) (\mu-\lambda)^{\frac32} v^{-1}$.
\end{Proposition}
\begin{proof}
    Suppose to the contrary that $\sigma \leq (D+o(1)) (\mu-\lambda)^{\frac32} v^{-1}$ 
    for some constant $D < 1$.
    By Lemma \ref{lem:evs_E}, $\nu_i \leq (D+o(1)) (\mu-\lambda)^{\frac32}$ 
    for an eigenvalue $\nu_i$ of $E$.
    Let $u_i$ be an eigenvalue of $A$.
    By Lemma \ref{lem:asrg_ev2} (i), if $u_i$ has positive form, then
    \begin{align*}
        u_i &= (1+o(1)) \left( \tfrac{k-\mu}{\mu-\lambda} + D \sqrt{\mu-\lambda} \right) 
        \leq (D +o(1)) \sqrt{\mu - \lambda} \eqcolon r.
    \end{align*}
    By Lemma \ref{lem:asrg_ev2} (ii), if $u_i$ has negative form, then 
    \begin{align*}
        u_i &= -(1+o(1)) (\mu-\lambda) \eqcolon s.
    \end{align*}
    Hence, $r^2 = -(D^2+o(1)) s$.
    By Proposition \ref{prop:krein_regular},
    \begin{align*}
        0 &\leq (s+r^2)v + 2(k-r)(r-s) \\
        &= (1-D^2+o(1)) sv - 2sk = (1-D^2+o(1)) s v < 0.
    \end{align*}
    This is a contradiction, so $\sigma \geq (1+o(1)) (\mu-\lambda)^{\frac32} v^{-1}$.
\end{proof}

We call a $1$-walk-regular graph {\it positive-$\alpha$-flat}
if all its positive eigenvalues are $\alpha$-flat.
Recall from the discussion at the end of \S\ref{sec:krein}
that for $k = o(v)$ it is natural to assume that 
a $1$-walk-regular graph is positive-$1$-flat.

\begin{Proposition}[Krein Bound, Positive-$1$-Flat $1$-Walk-Regular Approximately Strongly Regular Graphs]\label{prop:krein_bnd_for_asrg_1walk}
    Consider a family of positive-$1$-flat 
    $1$-walk-regular approximately strongly regular graphs 
    with $\mu > \lambda$, $k=o(v)$, and $k = o((\mu-\lambda)^{\frac32})$.
    Then $\sigma \geq (1+o(1)) (\mu - \lambda)^{\frac54} \cdot  v^{-\frac34}k^{\frac12}$.
\end{Proposition}
\begin{proof}
    Let $K_1$, $K_2$, and $L$ be as in Proposition \ref{prop:krein_1walk}
    where we use the estimate from Equation \ref{eq:bnd_FiFj} for $L$.
    
    \medskip  
    Our plan is as follows: 
    Choose $r = \Theta(\frac{k-\mu}{\mu-\lambda})$.
    We suppose that $\sigma \leq (D+o(1)) (\mu - \lambda)^{\frac54} \cdot v^{-\frac34}k^{\frac12} $
    for some constant $D < 1$, so smaller than claimed.
    From this we show that we have an eigenvalue $s$ of size $-(1+o(1))(\mu-\lambda)$, so that 
    \begin{align*}
        0 &\leq (1-K_1)s + r K_2 + L
    \end{align*}
    yields a contradiction if $K_1 = o(1)$, $rK_2 = o(-s)$, 
    and $L + o(1) < (D^4+o(1)) (-s)$.
    
    \medskip 
    
    We have $B = \sum u_i^2 = \tr(A) = vk = (1+o(1)) vk$.
    Note that in this sum we can ignore eigenvalues
    $u_i$ with $|u_i| = O(\frac{k-\mu}{\mu-\lambda})$ as these
    contribute at most $O(\frac{v k^2}{(\mu-\lambda)^2}) = o(vk)$ 
    to it (here we use $k-\mu = o((\mu-\lambda)^{\frac32})$).
    An eigenvalue $u_i$ of positive form 
    with associated eigenvalue $\nu_i$ of $E$ with $\nu_i = O(k)$ satisfies,
    by Lemma \ref{lem:asrg_ev2} (i),
    \begin{align*}
        u_i &= (1+o(1)) \tfrac{k-\mu + \nu_i}{\mu-\lambda} = O(\tfrac{k-\mu}{\mu-\lambda}).
    \end{align*}
    Hence, an eigenvalue $u_i$ of positive form 
    either has a significant contribution from $E$,
    that is $\nu_i = \omega(k)$,
    or does not contribute to $\sum u_i^2$.
    By Lemma \ref{lem:evs_E}, we find 
    \begin{align}
        B \coloneq \sum u_i^2 \leq \tfrac{v^2 \sigma^2}{(\mu-\lambda)^2} 
          \leq (D^2+o(1)) (\mu-\lambda)^{\frac12} v^{\frac12} k.\label{bnd:B}
    \end{align}
    We find, using Equation \eqref{bnd:B},
    \begin{align*}
        K_1 &= \tfrac{2}{vk} B \leq (2D^2+o(1)) (\mu-\lambda)^{\frac12} v^{-\frac12} = o(1), \text{ and } \\
        L &= \tfrac{1}{vk^2} B^2 + \tfrac{2}{v} B \leq (D^4+o(1)) (\mu-\lambda) = (D^4+o(1)) (-s).
    \end{align*}
    The Cauchy-Schwarz inequaltiy applied to Equation \eqref{bnd:B} 
    shows that $\left( \sum u_i \right)^2 \leq (D^2+o(1)) (\mu-\lambda)^{\frac12} v^{-\frac12} k$.
    Hence, 
    \[
     rK_2 \leq (D+o(1)) (\mu-\lambda)^{-\frac34} v^{-\frac14} k^{\frac32} = o(-s).
    \]
    
    \medskip 
    
    It remains to show that there exists an eigenvalue $s$ as asserted.
    Only considering $u_i$ with $\nu_i = \omega(k)$,
    so we are in one of the cases (i) or (vi) of Lemma \ref{lem:asrg_ev2}, 
    we see as in Equation \eqref{bnd:B} that
    \begin{align*}
     \sum (\mu-\lambda)^2 u_i^2 \leq (D^2+o(1)) (\mu-\lambda)^{\frac12} v^{\frac12} k.
    \end{align*}
    Hence, the contribution of $E$ to the sum $\sum u_i^2$
    is bounded by $(D^2+o(1)) (\mu-\lambda)^{\frac12} v^{\frac12} k$.
    Hence, eigenvalues without a significant
    contribution from $E$ account for at least
    $(1 + o(1)) vk$ of $\sum u_i^2 = (1+o(1))vk$.
    We saw earlier that such $u_i$ have negative form.
    Hence, by Lemma \ref{lem:asrg_ev2} (ii), $u_i = -(1+o(1)) (\mu - \lambda)$
    for some $u_i$.
\end{proof}

In light of \eqref{eq:exp_flat}, it might be more reasonable 
(at least if one is an optimist) 
to assume that the graph $O(\frac{k}{v})$-flat.
Then in the same notation as the previous 
proof and the argument from \eqref{eq:bnd_FiFj}, we have 
\[ 
 L \leq \frac{1+o(1)}{v^2k} \left( \sum_{i \in \cI} m_i u_i^2 \right)^2.
\]
Let us also state the absolute bound for this case.

\begin{Proposition}[Krein Bound, Positive-$O(\frac{k}{v})$-Flat $1$-Walk-Regular Approximately Strongly Regular Graphs]\label{prop:krein_bnd_for_asrg_1walk2}
    Consider a family of positive-$O(\frac{k}{v})$-flat 
    $1$-walk-regular approximately strongly regular graphs 
    with $\mu > \lambda$, $k=o(v)$, and $k = o((\mu-\lambda)^{\frac32})$.
    Then $\sigma = \Theta((\mu - \lambda)^{\frac54} \cdot  v^{-\frac12}k^{\frac14})$.
\end{Proposition}

\subsection{An Absolute Bound}

\begin{Proposition}[Absolute Bound, Approximately Strongly Regular Graphs]\label{prop:absolute_bnd_for_asrg}
    Consider a family of approximately strongly regular graphs such that 
    $\lambda > \mu$ and $\sqrt{v} \cdot k = o((\lambda-\mu)^2)$.
    Then $\sigma \geq (\tfrac13+o(1)) \frac{k}{v}$.
\end{Proposition}
\begin{proof}
    Our plan for applying Proposition \ref{prop:abs_reg} is as follows: 
    We can ignore $f_1$ as it is a minor order term.
    We suppose that $\sigma \leq (D + o(1)) \frac{k}{v}$
    for some constant $D<\frac13$.
    Put $r = (1+o(1)) (\lambda-\mu)$,
    $s = -(1-D+o(1)) \frac{k-\mu}{\lambda-\mu}$, and 
    $\eps = (2D+o(1)) \frac{k-\mu}{\lambda-\mu}$.
    As $D < \frac13$, we have that $s^2+s > \eps^2$.
    
    \medskip 
    
    By Lemma \ref{lem:evs_E}, any $\nu_i$ satisfies
    $|\nu_i| \leq v \sigma \leq (D+o(1)) k$.
    If an eigenvalue $u_i$ has positive form, then,
    by Lemma \ref{lem:asrg_ev2} (iii),
    \begin{align*}
        u_i = (1+o(1))(\lambda-\mu) = r.
    \end{align*}
    If an eigenvalue $u_i$ has negative form, then,
    by Lemma \ref{lem:asrg_ev2} (iv),
    \begin{align*}
        s-\eps = -(1+o(1)) \tfrac{k-\mu+Dk}{\lambda-\mu} 
        \leq u_i \leq -(1+o(1)) \tfrac{k-\mu-Dk}{\lambda-\mu} = s.
    \end{align*}
    
    To apply Proposition \ref{prop:abs_reg},
    it remains to determine $f_2$, that is we need to bound the number
    of restricted eigenvalues in $[r, k]$.
    We already saw that all such $u_i$ are of positive form.
    Using $\sum u_i^2 = \tr(A^2) = vk$, we see that there
    are at most $(1+o(1)) \frac{vk}{(\lambda-\mu)^2}$
    eigenvalues $u_i$ of positive form in this interval.
    Using $\sqrt{v} \cdot k = o((\lambda-\mu)^2)$, 
    we see $\frac{vk}{(\lambda-\mu)^2} = o(\sqrt{v})$.
\end{proof}

For instance, for $v = (1+o(1)) q^{11}$, $k = (1+o(1)) q^9$,
$\lambda = (1+o(1)) q^8$, and $\mu = (1+o(1)) q^7$,
Proposition \ref{prop:absolute_bnd_for_asrg} shows 
$\sigma \geq (\tfrac13 + o(1)) q^{-2}$.

\subsection{Cocliques}

\begin{Proposition}[Inertia Bound for ARSGs]\label{prop:inertia_asrg}
    Consider a family of approximately strongly regular graphs $\Gamma$
    such that $k = o(v)$, $k = o(|\lambda-\mu|^2)$.
    Then a coclique of $\Gamma$
    has at most size 
    \[ 
        (1+o(1)) \left( \frac{vk}{(\mu-\lambda)^2} + \frac{v^2\sigma^2}{k^2}\right).
    \]
\end{Proposition}
\begin{proof}
    Suppose without loss of generality that $\mu > \lambda$.
    We want to apply the inertia bound from \S\ref{sec:inertia}
    with the adjacency matrix,
    so we count negative eigenvalues.
    We need to count the number $g_1$ 
    of eigenvalues $u_i$ with $\nu_i = \Omega(k)$, 
    as these can be negative by Lemma \ref{lem:asrg_ev2} (i) and (vi),
    and we need to count the number $g_2$ of eigenvalues of 
    negative form with $\nu_i = O(k)$.
    
    For $g_1$, Lemma \ref{lem:evs_E} shows 
    $g_1 k^2 \leq v^2 \sigma^2$.
    For $g_2$, note that all eigenvalues
    considered have size $(1+o(1)) (\mu -\lambda)$,
    so $(1+o(1)) g_2 (\mu-\lambda)^2 \leq \sum u_i^2 = \tr(A^2) = vk$
    yields the claim.
\end{proof}

\section{Examples}

Clearly, strongly regular graphs provide plenty of examples
for approximately strongly regular graphs with $\sigma = 0$.
Let us present
examples with small, but nonzero $\sigma$.

\subsection{Very Small Examples}

We list, using the classification
of small regular graphs \cite{Meringer1999}, 
the number of connected graphs with smallest $\sigma$ for given 
$v$ and $k$. The last column contains a common 
name or a structure description of the 
automorphism group.

{\medskip 
\noindent
\begin{tabular}{ccccccl}
$v$ & $k$ & $\lambda$ & $\mu$ & $\sigma$ & nr & remarks \\
\hline
$8$  & $3$ & $0$ & $1.5$ & $0.5$ & $1$ & $D_8$ \\
$10$ & $3$ & $0$ & $1$ & $0$ & $1$ & Petersen graph, $NO^{-\perp}_{3,5}$ \\
$12$ & $3$ & $0$ & $0.75$ & $\sim0.43$ & $2$ & $D_8$, $D_9$ \\
$14$ & $3$ & $0$ & $0.8$ & $\sim0.49$ & $9$ & \\
$16$ & $3$ & $0.625$ & $0.34375$ & $\sim0.48$ & $2$ & $D_6$, $D_9$ \\
$18$ & $3$ & $0.\overline{6}$ & $0.3\overline{571428}$ & $\sim0.47$ & $2$ & $D_6$, $S_3^2 \rtimes C_2$ \\
$20$ & $3$ & $0.3$ & $0.31875$ & $\sim0.47$ & $5993$ & \\
$22$ & $3$ & $0.\overline{27}$ & $0.2\overline{87}$ & $\sim0.45$ & $86977$ & \\
$9$  & $4$ & $1$ & $2$ & $0$ & $1$ & ${\rm Paley}(9)$ \\
$10$ & $4$ & $0.75$ & $1.8$ & $\sim0.43$ & $1$ & $D_5$ \\
$11$ & $4$ & $1.\overline{09}$ & $1.\overline{27}$ & $\sim0.44$ & $1$ & $C_2^2 \times S_3$ \\
$12$ & $4$ & 1 & $1.\overline{142857}$ & $0.41$ & $1$ & $C_2 \times D_4$ \\
$13$ & $4$ & $0.\overline{692307}$ & $1.\overline{153846}$ & $\sim0.46$ & $1$ & $D_8$ \\
$14$ & $4$ & $0.32\overline{142857}$ & $1.\overline{190476}$ & $\sim0.47$ & $2$ & $\mathrm{id}$, $C_2^2$\\
$15$ & $4$ & $0.1$ & $1.16$ & $\sim0.37$ & $1$ & $D_6$ \\
$16$ & $4$ & $0$ & $1.\overline{09}$ & $0.36$ & $1$ & $C_2^4 \rtimes C_2$ \\
$12$ & $5$ & $0.7$ & $2.75$ & $\sim0.46$ & $1$ & $S_3^2$ \\
$14$ & $5$ & $1.0\overline{285714}$ & $1.\overline{857142}$ & $\sim0.45$ & $1$ & $C_2 \times D_4$ \\
$13$ & $6$ & $2$ & $3$ & $0$ & $1$ & ${\rm Paley}(13)$ \\
\end{tabular}\par}

\subsection{Some Examples from Literature}\label{sec:list_examples}

Various examples for small $\sigma$ occur in the literature.
Here we list some.

\begin{enumerate}[(i)]
 \item In \cite{RX2007} Radziszowski and Xiaodong
 describe an approximately strongly regular graph
 with parameters $(127, 42, 11, \mu; \sigma)$
 where $\mu \in [14, 16]$ and $\sigma \leq 1$.
 \item In \cite{BT1981} Bollob\'{a}s and Thomason construct an approximately
 strongly regular graph with parameters 
 $(2^r, 2^{r-1}-1, \lambda, \mu; 1)$ where
 $\lambda \in [2^{r-2}-2, 2^{r-2}-1]$
 and $\mu \in [2^{r-2}-1, 2^{r-2}]$.
 \item In \cite[Theorem 4]{SDPJ2016} Shi, Dong, Petersen, and Johansson
 show that certain graphs related to quantum networks
 are edge-regular approximately strongly regular graphs with 
 parameter $(v, k, n-2, \mu; \sigma)$ where $\mu \in [0,1]$
 and $\sigma \leq 1$.
 \item An edge-regular graph with parameters $(v, k, \lambda)$
 is  {\it quasi-strongly regular graph} with parameters 
 $(v, k, \lambda; \mu_1, \mu_2)$ if 
 $\mu_{ab} \in \{ \mu_1, \mu_2 \}$ for all nonadjacent 
 vertices $a,b$. For $\mu_1 < \mu_2$, 
 it is approximately strongly regular
 with parameters $(v, k, \lambda, \mu; \sqrt{\mu_2-\mu_1})$
 where $\mu \in [\mu_1, \mu_2]$, cf. \cite{Goldberg2006}.
 \item A $k$-regular graph of order $v$ is 
 a {\it Deza graph} with parameters $(v, k, \beta, \alpha)$,
 where $\alpha \leq \beta$,
 if $\lambda_{ab}, \mu_{ab} \in \{ \alpha, \beta \}$
 for all vertices $a,b$.
 It is approximately strongly regular with 
 parameters $(v, k, \lambda, \mu; \sqrt{\beta-\alpha})$
 where $\lambda, \mu \in [\alpha, \beta]$, cf. \cite{GS2021}.
 \item A random $k$-regular graph of order $v$, cf. \cite[\S2.4]{Bollobas2001},
 is an approximately strongly regular graph 
 with parameters (roughly) 
 $(v, k, \frac{k^2}{v}, \frac{k^2}{v}; \frac{k}{\sqrt{v}})$.
 \item Let $\Gamma$ be the intersection graph of $d$-spaces in $\Ff_q^n$,
 that is the graph with the $d$-spaces 
 of $\Ff_q^n$ as vertices,
 two adjacent if they meet nontrivially, 
 cf. \cite[\S1.2.4]{BvM}. 
 Let denote the number
 of $b$-spaces in $\Ff_q^a$ by $\gauss{a}{b}_q$. 
 It is easy to see that 
 for $n \gg 2d$ and $q \rightarrow \infty$,
 $\Gamma$ is approximately
 strongly regular with parameters 
 $(\gauss{n}{d}, \gauss{n}{d} - q^{d^2} \gauss{n-d}{d}, \lambda, \mu; 1)$,
 where $\lambda = \Theta( \gauss{n-1}{d-1} )$ and $\mu = \Theta(q^{2(d-1)} \gauss{n-2}{d-2})$.
 \item Consider a $2$-$(V, K, \Lambda)$ design $\mathcal{D}$, cf. \cite[\S6.2]{BvM}.  
 Let $\Gamma$
 be the graph with the blocks of $\mathcal{D}$ as vertices,
 two adjacent if they intersect.
 If $\Gamma$ is regular, then 
 $\Gamma$ is approximately regular with 
 parameters $(v, k, \lambda, \mu; \sigma)$
 where $v \sim \frac{V^2}{K^2} \Lambda$,
 $k \sim V\Lambda$, $\lambda \sim \frac{V\Lambda}{K}$,
 $\mu \sim K^2 \Lambda$, and $\sigma \rightarrow 0$
 if $K, \Lambda$ are constant and $V \rightarrow \infty$.
 Due to the seminal work by Keevash \cite{KeevashDesigns}, these exist if
 $V$ is sufficiently large and some divisibility 
 conditions are satisfied, but we do not know 
 if we can guarantee that each block is disjoint 
 to the same number of blocks, that is the 
 regularity of $\Gamma$.
 \item It is well-known that if one has equality 
 in the relative (or special) bound 
 for equiangular lines in $\Rf^d$,
 then one obtains a strongly regular graph
 from the Gram matrix of the set of 
 equiangular lines, cf. \cite[\S8.14]{BvM}.
 Similarly, if one is close, then 
 the graph from the Gram matrix
 has its spectrum concentrated at two values.
 This is used often, for instance 
 recently in \cite{GS2022,GSS2022}.
 Thus, if the graph is regular, then it 
 is approximately strongly regular with small $\sigma$.
 We did not estimate $\sigma$.
\end{enumerate}


\subsection{Orthogonality Graphs}\label{sec:orth_graphs}

Our application in \S\ref{sec:Ktfree}
is partially motivated by the construction described here.

For $q$ an odd prime power, let $V$ be the $n$-dimensional vector spaces over $\Ff_q$, 
the finite field with $q$ elements.
As $q$ is odd, $\Ff_q$ contains $\frac{q-1}{2}$ (nonzero) squares and
$\frac{q-1}{2}$ nonsquares.
Put $\gamma = 1$ if $q \equiv 1\pmod{4}$
and $\gamma = -1$ if $q \equiv 3\pmod{4}$.

A {\it quadratic form} over $\Ff_q$ is a map $Q: V \rightarrow \Ff_q$
such that $Q(\alpha v) = \alpha^2 Q(x)$ for all $\alpha \in \Ff_q$
and $x \in V$ and the function $B(x, y) \coloneq Q(x+y) - Q(x) - Q(y)$ is bilinear.
We can find an $(n \times n)$-matrix $M$ such that $Q(x) = x^T M x$.
We say that $Q$ is {\it nondegenerate} if $\det(Q) \coloneq \det(M) \neq 0$.
From now on we assume that $Q$ is nondegenerate.
For $x \in V$ nonzero, call $\<x\>$ a {\it point}.
Call a point $\<x\>$ {\it singular} when $Q(x) = 0$.
Cf. \S2 and \S3 in \cite{BvM} and \S11 in \cite{Taylor91}.

If $n=2m+1$ is odd, then there is only one choice for $Q$ up to isomorphism.
The set of nonsingular points $\< x \>$ splits into two parts of
sizes $\frac12 q^m (q^m+\eps)$ for $\eps \in \{ -1, 1\}$.
Here $\eps$ depends on $Q(x)$ being a (nonzero) square or a nonsquare.
Let $NO^{\eps\perp}_{n,q}$ denote the graph
with one of the parts as vertices,
two vertices $x,y$ adjacent when they are orthogonal, that is $B(x, y) = 0$.
We identify $\eps=1$ with $+$ and $\eps=-1$ with $-$,
so we write $NO^{+\perp}_{n,q}$ and $NO^{-\perp}_{n,q}$.
In \cite{BvM} these graphs are called $NO^{\eps\perp}_n(q)$
for $q \in \{ 3, 5 \}$.

The automorphism group of $NO^{\eps\perp}_{n,q}$ acts 
transitively on cliques of a given size
as the corresponding orthogonal group acts transitively
on tuples of pairwise orthogonal points of the same type.
In particular, $NO^{\eps\perp}_{n,q}$ is edge-regular with parameters
\begin{align*}
    & v = \tfrac12 q^m (q^m+\eps), && k = \tfrac12 q^{m-1} (q^m-\eps),
    && \lambda = \tfrac12 q^{m-1}(q^{m-1}+\gamma\eps).
\end{align*}
For $q=3,5$, the graph $NO^{\eps\perp}_{n,q}$ is 
strongly regular with $\mu = \frac12 q^{m-1} (q^{m-1}-\eps)$.
Standard counting for quadratic forms shows that
\begin{align*}
    \tfrac12 q^{m-1}(q^{m-1}-1) \leq \mu_{xy} \leq \tfrac12 q^{m-1}(q^{m-1}+1).
\end{align*}
Hence, for fixed $n$ and $q \rightarrow \infty$, the graph $NO^{\eps\perp}_{n,q}$ 
is approximately strongly regular with $\sigma \leq (1+o(1))q^{m-1}$.

\bigskip 

If $n=2m$ is even, then there are two choices for $Q$
up to isomorphism, depending on if $Q$ is of elliptic (put $\eps =-1$) 
or hyperbolic type (put $\eps = 1$). We can distinguish them 
by the number of singular points which is
$\frac{(q^{m-1}+\eps) (q^{m}-\eps)}{q-1}$.
We can also distinguish them by $\det(Q)$ being a (nonzero) square or a nonsquare.
The nonsingular points $\<x\>$ split into two orbits of equal size
$\frac12 q^{m-1}(q^m-\eps)$ each,
depending on $Q(x)$ being a square or a nonsquare.
Let $NO^{\eps\perp}_{n,q}$ denote the graph
with one of the parts as vertices,
two vertices adjacent when orthogonal, that is $B(x, y) = 0$.
In \cite{BvM} these graphs are called $NO^{\eps}_n(q)$ for $q=3$.

As for $n$ odd, the automorphism group of $NO^{\eps\perp}_{n,q}$ 
acts transitively on cliques.
In particular, it is edge-regular with parameters
\begin{align*}
    & v = \tfrac12 q^{m-1} (q^m - \eps), && k = \tfrac12 q^{m-1}(q^{m-1} - \gamma \eps), 
    && \lambda = \tfrac12 q^{m-2} (q^{m-1} + \gamma\eps).
\end{align*}
For $q=3$, the graph $NO^{\eps\perp}_{n,q}$ is strongly regular 
with $\mu = \frac12 q^{m-1}(q^{m-2}+ \eps)$.
Standard counting for quadratic forms shows that
\begin{align*}
    \tfrac12 q^{m-1}(q^{m-2}-1) \leq \mu_{xy} \leq \tfrac12 q^{m-1}(q^{m-2}+1).
\end{align*}
Hence, for fixed $n$ and $q \rightarrow \infty$, the graph $NO^{\eps\perp}_{n,q}$ 
is approximately strongly regular with $\sigma \leq (1+o(1))q^{m-1}$.

\bigskip 

There is the following tower of graphs 
(see \cite[p. 89]{BvM} for the case $q=3$):
Let $NO^{\eps\perp}_{n,q}(x)$ denote the induced subgraph 
on the neighborhood of $x$ in $NO^{\eps\perp}_{n,q}$.
We find that $NO^{\eps\perp}_{2m+1,q}(x)$ is 
isomorphic to $NO^{\eps\perp}_{2m,q}$,
and that $NO^{\eps\perp}_{2m,q}(x)$ is 
isomorphic to $NO^{\gamma\eps\perp}_{2m-1,q}$.
The graph $NO^{+\perp}_{2,q}$
is edgeless if $\gamma = -1$, otherwise
it is the union of $\frac{q-1}{4}$ pairwise disjoint edges.
The graph $NO^{-\perp}_{2,q}$ is edgeless if $\gamma = 1$,
otherwise it is the union of $\frac{q+1}{4}$ pairwise disjoint edges.
By induction, we find that the clique number of $NO^{\eps\perp}_{n,q}$
for $n=2m+1$ or $n=2m$ is $n-1$ if $\gamma \eps = (-1)^m$
and $n$ if $\gamma \eps = -(-1)^m$.

\section{Some Applications}\label{sec:app}

\subsection{Large Caps}\label{sec:caps}

Let $n \geq 2$ and let $q$ be a prime power.
Consider a set of points $\cC$ in $\PG(n, q)$, the $n$-dimensional
projective space over $\Ff_q$.
We use that the number of points in 
$\PG(n, q)$ is $\frac{q^{n+1}-1}{q-1}$.
If no three points in $\cC$ are collinear,
then $\cC$ is called a {\it cap}.

For the regime of $q=3$ and $n \rightarrow \infty$,
the cap set problem recently gained much prominence due to the 
breakthrough result by Ellenberg and Gijswijt, see \cite{EG2017}.
Here we consider the regimes where $n$ is fixed and $q \rightarrow \infty$
as well as where $q$ is fixed and $n \rightarrow \infty$.
Note that \cite{EG2017} considers caps in $\Ff_q^n$, not $\PG(n, q)$, 
but this only changes bounds by constant factor.
We always assume that $q \geq 3$
as $q=2$ is trivial.
As the calculations for $n$ fixed require slightly more care
than those for $q$ fixed (but are essentially identical),
we will only include those.
It is easy to see that a cap has size at most $(1+o(1)) q^{n-1}$
for $n$ fixed and $q \rightarrow \infty$.
The largest known constructions for caps have
size $\Theta(q^{\lfloor \frac23 n \rfloor})$. 
This is tight for $n = 2,3$.
See \cite{Edel2004,EB1999} for constructions 
of caps for large $n$ or large $q$.

\medskip 

It is well-known that caps define graphs in various ways.
For a cap $\cC$ of $\PG(n, q)$, define an {\it associated graph} 
$\Gamma$ as follows: 
Consider $\Ff_q^{n+1}$ with $\PG(n, q)$ as hyperplane at infinity.
Take the vectors of $\Ff_q^{n+1}$ as vertices,
two distinct $a,b \in \Ff_q^{n+1}$ adjacent
if $\< a, b \>$ meets $\PG(n, q)$ in a point of $\cC$.
Put $t = |\cC|$.
It is well-known and easy to verify that
this defines an edge-regular graph with
$(v, k, \lambda) = (q^{n+1}, t(q-1), q-2)$.
An {\it exterior point} of $\cC$ is a point of $\PG(n, q)$
not in $\cC$ and a {\it secant} of $\cC$ is a line of $\PG(n, q)$
which meets $\cC$ in precisely two points.

If each exterior point lies on precisely the same number $h$
of secants, then we obtain a strongly regular graph
with $\mu = \frac{t(t-1)(q-1)^2}{q^{n+1}-t(q-1)+1}$.
See \S8.7.1(vi) in \cite{BvM}.
We can say the following about this case.

\begin{Lemma}\label{lem:srg_cap}
    Let $\cC$ be a cap of size $t$ in $\PG(n, q)$ such that each exterior point lies on
    a constant number of secants. Then
    \begin{align*}
        &t \leq (1+o(1)) q^{\frac34 n - \frac14} && (\text{as } q \rightarrow \infty),\\
        &t = O(q^{\frac34 n}) && (\text{as } n \rightarrow \infty).
    \end{align*}
\end{Lemma}
\begin{proof}
    We only prove the first part.
    Let us calculate the negative eigenvalue $s < 0$ of 
    the associated graph $\Gamma$.
    We find $s = -(1+o(1)) \mu$.
    One of the Krein conditions, Theorem \ref{thm:krein_srg}, 
    requires
    \begin{align*}
        0 &\leq 1 + \frac{s^3}{k^2} - \frac{(s+1)^3}{(v-k-1)^2}\\
        &= 1 - (1+o(1)) \frac{(t^2 q^{-n+1})^3}{(tq)^2}
        = 1 - (1+o(1)) \frac{t^4}{q^{3n-1}}.    
    \end{align*}
    Hence, $t \leq (1+o(1)) q^{\frac34 n - \frac14}$.
\end{proof}

{\medskip \footnotesize 
We can obtain the same bound using the inertia bound 
(estimates for $q \rightarrow \infty$ only):
We find that the multiplicity $g$ of $s$ is 
$(1+o(1)) vk/s^2 = (1+o(1)) q^{3n}/t^3$.
As $\cC$ has a clique of size $t$,
we find $t^4 \leq (1+o(1)) q^{3n}$.
Lastly, the absolute bound shows $t \leq (1+o(1)) q^{\frac56 n + \frac16}$.
\par} 

\medskip 

How much can we weaken the condition on the exterior points
and secants? From now on let $h$ be the expected
number of secants through an exterior point $p$ of $\PG(n, q)$ and let $h_p$
the actual number of secants through $p$.

\begin{Lemma}\label{lem:var_h_p}
  Let $\cC$ be a cap of size $t$ in $\PG(n, q)$ with 
  an associated approximately strongly regular graph $\Gamma$ 
  with parameters $(v, k, \lambda, \mu; \sigma)$.
  Then
  \begin{align*}
   & \Var(h_p) = (\tfrac14+o(1)) \Var(\mu_{ab}) && (\text{as } q \rightarrow \infty),\\
   & \Var(h_p) = \Theta(\Var(\mu_{ab})) && (\text{as } n \rightarrow \infty).
  \end{align*}
\end{Lemma}
\begin{proof}
  We only show the assertion for $q \rightarrow \infty$.
  Let $M$ be as in the introduction, that is all pairs of nonadjacent vertices, 
  and let $\cD$
  denote the set of exterior points of $\cC$.
  Note that $|M| = q^{n+1} (q^{n+1} - t(q-1))$ and 
  that $|\cC| \leq (1+o(1)) q^{n-1}$ implies that $|\cD| = (1+o(1)) q^n$.
  If for two distinct nonadjacent
  vertices $a,b$ the line $\< a, b\>$ meets $\cD$ in $p$,
  then $2h_p$ is the number of common neighbors of $a$ and $b$.
  Hence,
  \begin{align*}
    \Var(\mu_{ab}) &= \frac{1}{M} \sum_{a \not\sim b} (\mu_{ab} - \mu)^2
    = \frac{1}{M} \sum_{p \in \cD} 
      \sum_{\substack{a \not\sim b,\\ \<a,b\> \cap \PG(n, q) = p}} (\mu_{ab} - \mu)^2\\
    &= \frac{1}{M} \sum_{p \in \cD} 4 \cdot q^{n+1} (q-2) \cdot (h_p - h)^2
    = (4+o(1)) q^{-n} \sum_{p \in \cD} (h_p - h)^2\\
    &= (4+o(1)) \frac{1}{|\cD|} \sum_{p \in \cD} (h_p - h)^2 
    = (4+o(1)) \Var(h_p). \qedhere
  \end{align*}
\end{proof}

\medskip 

\begin{Proposition}\label{prop:cap_bound_delta}
    For $n \geq 4$, 
    let $\cC$ be a cap of size $t \geq (2+o(1)) q^{\frac34 n}$ in $\PG(n, q)$ 
    and let $\cD$ denote its exterior points.
    \begin{enumerate}[(i)]
     \item If $\Var(h_{p}) \leq (\frac14+o(1)) \sigma^2$, then 
            $\sigma \geq (1+o(1)) t^{\frac32} q^{-n}$ (as $q \rightarrow \infty$).
     \item If $\Var(h_{p}) = \Theta(\sigma^2)$, then 
            $\sigma = \Omega(t^{\frac32} q^{-n})$ (as $n \rightarrow \infty$).
    \end{enumerate}
    If the associated graph $\Gamma$ is also postive-$1$-flat
    $1$-walk-regular, then we have the following.
    \begin{enumerate}[(i)]
     \item[(iii)] If $\Var(h_{p}) \leq (\frac14+o(1)) \sigma^2$, then 
            $\sigma \geq (1+o(1)) t^{3} q^{-2n+1}$ (as $q \rightarrow \infty$).
     \item[(iv)] If $\Var(h_{p}) = \Theta(\sigma^2)$, then 
            $\sigma = \Omega(t^{3} q^{-2n})$ (as $n \rightarrow \infty$).
    \end{enumerate}
\end{Proposition}
\begin{proof}
   Recall that from 
   $t = \Omega(q^{\frac34 n})$ and $n \geq 4$,
   we obtain $k = (q-1)t = o(\mu^{\frac32})$ and $\lambda = q-2 = o(\mu)$
   (as $\mu = (1+o(1)) t^2 q^{1-n}$).
   
   \smallskip 
   
   For the first part we apply the inertia bound
   from Proposition \ref{prop:inertia_asrg}.
    We already saw in the discussion on strongly regular graphs
    that we require $t \geq (1+o(1)) q^{\frac34 n}$
    for the first summand.
    What remains of Proposition \ref{prop:inertia_asrg} is
\[
    t \leq (1+o(1)) \frac{v^2\sigma^2}{k^2} = (1+o(1)) \frac{q^{2n+2} \sigma^2}{q^2t^2}.
\]
Rearranging for $\sigma$ yields $\sigma \geq (1+o(1)) t^{\frac32} q^{-n}$.
   
   \smallskip 
   For (iii) and (iv), use Proposition \ref{prop:krein_bnd_for_asrg_1walk}.
   We find 
   \[
        \sigma \geq (1+o(1)) (\mu - \lambda)^{\frac54} \cdot v^{-\frac34} k^{\frac12} = t^3 q^{-2n+1}.\qedhere 
   \]
\end{proof}

One can also use Proposition \ref{prop:krein_bnd_for_asrg}, 
but the resulting bounds on $\sigma$
are slightly worse than what is stated 
in Proposition \ref{prop:cap_bound_delta}.

\medskip 

Proposition \ref{prop:cap_bound_delta} 
implies the following.

\begin{Corollary}\label{cor:caps_bnds}
  For $n \geq 4$, let $\cC$ be a cap of size $t$ in $\PG(n, q)$.
  If $\Var(h_p) = o( q^{\frac14 n})$,
  then $t = O(q^{\frac34 n})$.
  If the associated graph is positive-$1$-flat $1$-walk-regular, then already
        $\Var(h_p) = o( q^{\frac12 n})$ implies
   $t = O(q^{\frac34 n})$.
   
   \smallskip 
   
  The above holds for $q \rightarrow \infty$ as well as $n \rightarrow \infty$.
\end{Corollary}

{\medskip \footnotesize 
For $t = \Theta(q^{n-1})$ (as $n \rightarrow \infty$),
we find $\sigma = \Omega(q^{n-2})$.
Using Proposition \ref{prop:krein_bnd_for_asrg_1walk2} instead 
of Proposition \ref{prop:krein_bnd_for_asrg_1walk} only yields 
a marginal improvement. For instance, for 
and $t = \Theta(q^{n-1})$ (as $n \rightarrow \infty$), 
we find $\sigma = \Omega(q^{n-2 + 0.25})$.

\smallskip 

If we assume that the setwise stabilizer of $\cC$ (a subgroup of $P\Gamma{}L(n, q)$)
acts transitively on $\cC$, then $\Gamma$ is $1$-walk-regular.
Maybe it is feasible to use Corollary \ref{cor:caps_bnds}
to show a bound of the form $O(q^{C n})$ for some $C < 1$
under some symmetry conditions.

\smallskip 

For $q=3$, Edel constructed caps of size $\Omega(2.21^n)$ \cite{Edel2004} 
and there is an upper bound of $o(2.76^n)$ by Ellenberg and Gijswijt \cite{EG2017}.
For the special cases of Corollary \ref{cor:caps_bnds}, we find an
upper bound of $o(2.28^n)$; 
for Lemma \ref{lem:srg_cap}, we also find $o(2.28^n)$.
In general it is known that 
$t \leq (1 - O(q^{-\frac12})) q^{n-1}$ (as $q \rightarrow \infty$),
cf. Table 4.4(ii) in \cite{HS2001}.\par} 

\subsubsection*{Explicit Bounds}

We can also find explicit bounds.
In the following we will demonstrate
this with some crude estimates.

\smallskip 

Suppose that we are looking for a cap of size $t$. 
We have a coclique of size $t$ in the 
associated graph, so 
we need at least $t$ nonpositive
eigenvalues. 
Let $f$ be the 
number of positive eigenvalues, 
and let $g=g_1+g_2$ be the 
number of negative eigenvalues,
$g_1$ of negative form (so they are at most $-\frac12 (\lambda-\mu)$), 
$g_2$ of positive form.
Then we require (using standard trace arguments as before)
\begin{align*}
    &g_1 + g_2 \geq t, && k^2+g_1 \cdot \tfrac1{2^2} (\lambda-\mu)^2 \leq vk, 
        && g_2(k-\mu)^2 \leq v(v-k-1)\sigma^2.
\end{align*}
For instance,
by \cite{HS2001}, a cap in $\PG(10, 3)$ has 
size at most $10937$. 
The largest known cap in $\PG(10, 3)$ 
has size 2744 \cite{EB2001}.
Put $t = 10937$. 
We find $(v,k,\lambda,\mu) = (3^{11}, 2 \cdot 10937, 1, \frac{29901758}{19409})$.
Then $g_1 \leq 5731$, so $g_2 \geq 10937-5731 = 5206$.
Hence, $5206 \cdot (k-\mu)^2 \leq v(v-k-1) \sigma^2$.
We obtain that $\sigma \geq 8.84$.

\subsection{Optimally Pseudorandom Clique-free Graphs}\label{sec:Ktfree}

A $k$-regular graph $\Gamma$ of order $v$
is called {\it optimally pseudorandom}
if the second largest eigenvalue in absolute value
of its adjacency matrix is in $O(\sqrt{k})$, cf. \cite{KS2006}.

\begin{Proposition}[Alon and Krivelevich, \cite{AK1997}]\label{prop:AK}
    Let $\Gamma$ be a $K_m$-free $k$-regular graph of order $v$
    with smallest eigenvalue $s$ such that $-s = O(\sqrt{k})$.
    Then 
    \[
      k = O(v^{1-\frac{1}{2m-3}}).
    \]
\end{Proposition}

This bound is tight for $m=3$ due to a construction by Alon \cite{Alon1994}.
Alon and Krivelevich gave an example with $k = \Theta(v^{1-\ssfrac{1}{m-2}})$ \cite{AK1997}.
The author noticed that there is 
a well-known construction with $k = \Theta(v^{1-\ssfrac{1}{m-1}})$ \cite{BIP2020}.
These are the graphs $NO^{\eps\perp}_{m,q}$ from \S\ref{sec:orth_graphs}
with clique number $m-1$.

\medskip 

The Ramsey number $R(m, n)$ is the largest number such that
there exists a graph on $R(m,n)-1$ vertices without a clique
of size $m$ or a coclique of size $n$.
Ajtai, Koml\'{o}s, and Szemer\'{e}di \cite{AKS1980}
and Bohman and Keevash \cite{BK2010} proved
\begin{align*}
    &\Omega\left( \tfrac{n^{\frac{m+1}{2}}}{(\log n)^{\ssfrac{m+1}{2} - \ssfrac{1}{m-2}}} \right) 
    = R(m,n) = O\left( \tfrac{n^{m-1}}{(\log n)^{m-2}} \right) 
    &&\text{(as $n \rightarrow \infty$)}.
\end{align*}
Recently, Mubayi and Verstra\"{e}te showed 
in \cite{MV2019} that if the upper bound
in Proposition \ref{prop:AK} is tight for some $m$, 
then $R(m,n) = \Omega(\frac{n^{m-1}}{(\log n)^{2m-4}})$,
nearly matching the upper bound. 
Their result also implies that if 
one finds a construction with 
$k = \Omega(v^{1-\ssfrac{1}{m+\eps}})$ for some $\eps>0$, 
then 
\begin{align*} 
&R(m,n) = \Omega\left(\tfrac{n^{\frac{m+\eps+1}{2}}}{(\log n)^{m+\eps+1}}\right) &&\text{(as $n \rightarrow \infty$)},
\end{align*}
which would improve the lower bound on $R(n, m)$.
Our technique here cannot show anything better
than $k = O(v^{1 - \ssfrac{1}{m+1}})$.

\medskip 

For the remainder of the section
consider an optimally pseudorandom 
$K_m$-free $k$-regular graph $\Gamma$ of order $v$,
smallest eigenvalue $s$, and second largest eigenvalue $r$, where $m \geq 3$.
Let $Y$ be a clique of size $i$ of $\Gamma$.
Let $\Gamma(Y)$ be the induced subgraph on the common neighborhood 
of $Y$. Let us define the following properties 
for $\Gamma_i \coloneq \Gamma(Y)$.
\begin{enumerate}
 \item[(P1)] The graph $\Gamma_i$ has $v_i$ vertices and 
            $\frac12 v_i k_i$ edges, where
        \begin{align*}
            &k_i \geq (1+o(1)) k \left( \tfrac{k}{v} \right)^i,  &&
            (1+o(1)) \tfrac{k}{v} \leq \tfrac{k_i}{v_i} = o(1).
        \end{align*}
  \item[(P2)] The graph $\Gamma_i$ is approximately strongly 
        regular with parameters $(v_i, k_i, \lambda_i, \allowbreak \mu_i; \sigma_i)$ and 
        its smallest eigenvalue $s_i$ satisfies
        \[
            -s_i = O(\mu_i-\lambda_i).
        \]
\end{enumerate}

Clearly, the graphs $NO^{\eps\perp}_{m,q}$ with clique number $m-1$
satisfy (P1). Furthermore, as mentioned in \S\ref{sec:orth_graphs},
the automorphism group of $NO^{\eps\perp}_{m,q}$ acts transitively
on cliques of a given size.
Hence, $\Gamma_i$ is regular.
The graphs $NO^{\eps\perp}_{m,q}$ with clique number $m-1$ often have property (P2).
We will see that property (P2) follows from (P1) for $i=m-3$
when $\Gamma_{m-3}$ is regular as $\Gamma_{m-3}$
is triangle-free, so $\lambda_{m-3} = 0$. More generally,
some $\Gamma_{i}$ must have property (P2) as $\lambda_i \leq (D+o(1)) \ssfrac{k_i^2}{v_i}$
has to occur for some $D<1$ for some $i$.

\medskip 

Let us state the expander-mixing lemma for the special case of only one set,
see the proof of Proposition 1.1.6 in \cite{BvM}.

\begin{Lemma}[Expander-Mixing Lemma, Variant]\label{lem:eml}
    Let $Y$ be a set of vertices of size $y$ of 
    a $k$-regular graph $\Gamma$ of order $v$
    with second largest eigenvalue $r$ and smallest eigenvalue $s$. 
    Then the number $e$ of edges in the induced subgraph on $Y$ satisfies
    \[
        \tfrac12 y \left( \tfrac{y(k-s)}{v} + s\right) \leq e \leq \tfrac12 y \left( \tfrac{y(k-r)}{v} + r\right).
    \]
\end{Lemma}

\begin{Lemma}\label{lem:clique-free_paras}
  Let $0 \leq i \leq m-3$.
  \begin{enumerate}[(i)]
  \item If $k = \omega(v^{1-\frac{1}{2i{+}1}})$, then
        there exists a $\Gamma_i$ with property (P1).
   \item If $\Gamma_i$ is regular and has property (P1), then
  \[
     -s_i = \Omega\left(k_i \left( \tfrac{k_i}{v_i} \right)^{m-i-2}\right).
  \]
  \end{enumerate}
\end{Lemma}
\begin{proof}
    First we show (i).
    For this, let $k_i$ 
    denote the average degree of $\Gamma_i$.
    Clearly, the claim is true for $i=0$.
    The condition $k = \omega(v^{1 - \frac{1}{2i{+}1}})$
    is equivalent to
    \begin{align}
        (1+o(1)) k \left( \tfrac{k}{v} \right)^{i} = \omega(\sqrt{k}).\label{eq:ki_bnd}
    \end{align}
    Suppose that the claim is true for $\Gamma_{i-1} = \Gamma(Y)$
    for some clique $Y$ of size $i-1$.
    Let $a$ be a vertex of $\Gamma_{i-1}$ of degree at least $k_{i-1}$.
    Take $\Gamma_i = \Gamma(Y \cup \{ a \})$.
    By Lemma \ref{lem:eml} applied to $\Gamma$, using Equation \eqref{eq:ki_bnd}, 
    the average degree $k_i$ of a vertex in $\Gamma_{i}$ satisfies
    \begin{align*}
        (1+o(1)) k \left(\tfrac{k}{v} \right)^{i} \leq \tfrac{k_{i-1}(k-s)}{v} + s \leq
        k_i \leq \tfrac{k_{i-1}(k-r)}{v} 
        + r = (1+o(1)) k_{i-1} \cdot \tfrac{k}{v}.
    \end{align*}
    This shows property (P1) for $\Gamma_{i}$.
    
    \medskip
    
    Next we show (ii).
    For some $j$ with $i \leq j \leq m-3$ we require 
    that $\lambda_j < (D+o(1)) \frac{k_j^2}{v_j}$ for some constant $D < 1$
    as $\Gamma$ is $K_m$-free.
    Similarly to the above, for the first $j$ 
    for which this occurs, 
    we find a $\Gamma_j$ with
    \[
      k_j = (1+o(1)) k_i \left(\tfrac{k_i}{v_i} \right)^{j-i}.
    \]
    By Lemma \ref{lem:eml}, applied to the regular graph $\Gamma_i$,
    \[
        (D+o(1)) \tfrac{k_j^2}{v_j} \geq \lambda_{j} \geq 
        (1+o(1)) k_i \left(\tfrac{k_i}{v_i} \right)^{j-i+1} + s_i.
    \]
    In the worst case is $j=m-3$ which yields the claim.
\end{proof}

\begin{Proposition}\label{prop:opt_pseudo1}
    Suppose that $\Gamma$ has 
    $k = \allowbreak \omega(v^{1-\frac{1}{3m{-}2i{-}5}})$.
    Furthermore, suppose that $\Gamma_i$
    has property (P1), and 
    that $\Gamma_i$ has property (P2) or $i=m-3$.
    If $\Gamma_i$ is regular, then $\sigma_i = \Omega(k^{\frac12} \left( \tfrac{k}{v} \right)^{\frac32 m - 2 - i})$.
    If $\Gamma_i$ is also positive-$1$-flat $1$-walk-regular,
    then $\sigma_i = \Omega(k \left( \tfrac{k}{v} \right)^{\frac54 m + \frac34 - \frac14 i})$.
\end{Proposition}
\begin{proof}
    If $\Gamma_i$ has property (P1) and is regular,
    then it is an approximately strongly regular graph 
    with parameters $(v_i, k_i, \lambda_i, \mu_i; \sigma_i)$
    with $\mu_i = (1+o(1)) \frac{k_i^2}{v_i^2}$ (as $k_i = o(v)$).
    If $\Gamma_{m-3}$ has property (P1), then $\lambda_{m-3} = 0$ and 
    $\mu_{m-3} = (1+o(1)) \frac{k_{m-3}^2}{v_{m-3}}$, 
    so $\Gamma_{m-3}$ has property (P2).
    
    \medskip 
    
    By Lemma \ref{lem:clique-free_paras}(ii) and property (P2), 
    \[
      \mu_i - \lambda_i = \Omega\left(k_i \left( \tfrac{k}{v} \right)^{m-i-2}\right).
    \]
    From $k = \omega(v^{1-\frac{1}{3m{-}2i{-}5}})$
    and $k_i \geq (1+o(1)) k \left(\frac{k}{v} \right)^i$, we obtain that
    $k_i = o(|\mu_i - \lambda_i|^{\frac{3}{2}})$.
    By Proposition \ref{prop:krein_bnd_for_asrg},
    \begin{align*}
        \sigma_i &\geq (1+o(1)) \frac{|\mu_i-\lambda_i|}{v_i}
        = \Omega\left(k^{\frac12} \left( \tfrac{k}{v} \right)^{\frac32 m - 2 - i}\right). \qedhere
    \end{align*}
    
    \smallskip 
    
    This shows the general case.
    The case with $\Gamma_i$ positive-$1$-flat $1$-walk-regular
    uses Proposition \ref{prop:krein_bnd_for_asrg_1walk} and 
    is otherwise similar.
\end{proof}

For $i=m-3$, we can also use the inertia bound.

\begin{Proposition}\label{cor:opt_Kt_bnds_inertia}
    Let $m \geq 5$. 
    Let $\Gamma_{m-3}$ be as in Proposition \ref{prop:opt_pseudo1}.
    If $\sigma_{m-3} = 
    \Omega\left(k^{\frac{1}{2}} \left( \tfrac{k}{v} \right)^{\frac12 m-\frac52}\right)$,
    then $k = O(v^{1 - \frac{1}{m+1}})$.
\end{Proposition}
\begin{proof}
    Apply Proposition \ref{prop:inertia_asrg}
    with $v = v_{m-3} = k_{m-2}$, $k = k_{m-3}$, and
    \[
     \mu-\lambda = \Omega\left( k \left( \tfrac{k}{v} \right)^{m-2} \right).
    \]
    Hence, with the chosen $\sigma$
    we cannot have a coclique of size $k_{m-3}$.
\end{proof}

For $i \leq \frac34 m - \frac32$, we have $k = \omega(v^{1-\frac{1}{2i{+}1}})$
in Proposition \ref{prop:opt_pseudo1}, so
we can apply Lemma \ref{lem:clique-free_paras}(i)
and see that there exists a $\Gamma_i$ with property (P1).
Hence, the case $i = \lfloor \frac34 m - \frac32 \rfloor$ 
is special.

\begin{Corollary}\label{cor:opt_Kt_bnds}
    Let $m \geq 5$. 
    Let $\Gamma_i$ be as in Proposition \ref{prop:opt_pseudo1}.
    \begin{enumerate}[(i)]
     \item 
        If $i=\frac34 m - \frac32$, 
        and $\sigma_{i} = o(k^{\frac12} \left( \tfrac{k}{v} \right)^{\frac34 m - \frac12})$,
        then $k = O(v^{1 - \frac{2}{3 m - 4}})$.
     \item If $i=m-3$ and $\sigma_{m-3} = o(k^{\frac12} \left( \tfrac{k}{v} \right)^{\frac12 m - \frac52})$, 
     then $k = O(v^{1 - \frac{1}{m+1}})$.
    \end{enumerate}
    If $\Gamma_i$ is also positive-$1$-flat $1$-walk-regular,
    then $\sigma_i = o(k \left( \tfrac{k}{v} \right)^{m-1})$, 
    respectively, $\sigma_i = o(k \left( \tfrac{k}{v} \right)^{\frac78 m-\frac{11}{8}})$
    suffice in (i), respectively, (ii).
\end{Corollary}


\subsection{SRGs as Counterexamples}

Glock, Janzer, and Sudakov ask in the conclusion of
\cite{GJS2021} for a family of clique-free 
strongly regular graphs with large $\lambda$ 
to disprove several conjectures in extremal
combinatorics. To our knowledge no such graph 
is known. Approximately strongly regular graphs
with small $\sigma$ are equally suitable for 
this task.

\section{Future Work}

There are countless results specific to
strongly regular graphs. Generalizing them
to approximately strongly regular graphs seems to
be a worthwhile endeavor. 

\smallskip 

Maybe one can improve the bounds given here:
our variant of the absolute bound,
Proposition \ref{prop:absolute_bnd_for_asrg}, is not very
satisfying compared to our Krein bounds.

{\medskip \footnotesize 
Bounds on approximately 
equiangular lines might be helpful here.
This is not a completely new topic, for instance,
constructions for almost equiangular lines 
were investigated in \cite{BO2015}.
\par} 

\smallskip 

While anything the author tries 
to construct will usually satisfy the conditions
of Proposition \ref{prop:krein_1walk} or
Proposition \ref{prop:krein_bnd_for_asrg_1walk}
(or be very close to it),
the setup seems overly technical and hard to verify
compared to Proposition \ref{prop:krein_regular}
and Proposition \ref{prop:krein_bnd_for_asrg}.
Maybe it can be simplified.
There is also the question if our results --
using the usual connections between caps,
strongly regular graphs, and linear codes -- 
has any interesting implications for coding theory.

\smallskip 

One can ask several existence questions, for instance:
\begin{enumerate}[(i)]
 \item For given $(v, k, \lambda, \mu)$, what is the smallest $\sigma$
 such that an approximately strongly regular graph 
 with parameters $(v, k, \lambda, \mu; \sigma)$ exists?
 \item For a given strongly regular graph with parameters 
 $(v, k, \lambda, \mu)$, what is the smallest $\sigma$
 such that an approximately strongly regular graph 
 with parameters $(v, k, \lambda, \mu; \sigma)$ exists
 which is not strongly regular?
\end{enumerate}

{\medskip \footnotesize 
Allen W. Herman suggested (ii) and some variants.
These questions and also some examples in \S\ref{sec:list_examples} 
suggest that a nonregular version of approximately 
strongly regular graphs might be interesting.

More formally, for a vertex $a$, let $k_a$ denote its degree.
Call a graph $\Gamma$ {\it nonregular
approximately strongly regular} with parameters $(v, k, \lambda, \mu; \sigma)$ 
the same as for an approximately strongly regular graph,
except that we no longer require $\Gamma$ to be $k$-regular,
just $k = \frac1v \sum_a k_a$ and 
$\Var(k_a) := \frac1v \sum_a (k_a - k)^2 \leq \sigma^2$.
It might be interesting to investigate 
how to make a nonregular approximately
strongly regular graph regular without 
changing its parameters too much.
\par}

\smallskip 

Our primary motivation for this document is to 
restrict the search space when looking for constructions
for specific extremal problems.
Maybe the techniques in this paper can be expanded
to obtain more general bounds on caps and
optimally pseudorandom clique-free graphs.
At the time of writing, the author holds the weak belief
that Corollary \ref{cor:caps_bnds} and 
Corollary \ref{cor:opt_Kt_bnds} state true upper
bounds for the respective general cases.

\bigskip
{\footnotesize
\paragraph*{Acknowledgments}
The author is supported by a 
postdoctoral fellowship of the Research Foundation -- Flanders (FWO).

\smallskip

The author thanks 
Andries E. Brouwer for several suggestions, particularly, 
for the term {\it approximately strongly regular graphs},
Noga Alon, Sebastian M. Cioab\u{a}, David Conlon,
Maar\-ten De Boeck, Shaun Fallat, Dion Gijswijt, 
Gary Greaves, Allen W. Herman, 
Sam Mattheus, Karen Meagher,
and Padraig \'{O} Cath\'{a}in for comments and remarks,
and Jacques Verstra\"ete
for hosting the author for the last two months in 2021
and discussing caps and clique-free graphs.\par}


\begin{thebibliography}{99}

\bibitem{AKS1980}
M. Ajtai, J. Koml\'{o}s, and E. Szemer\'{e}di, 
{\it A note on Ramsey numbers}, 
J. Combin. Theory Ser. A {\bf 29}(3) (1980) 354--360.

\bibitem{Alon1994}
N. Alon,
{\it Explicit Ramsey graphs and orthonormal labelings},
Electronic J. Combin. {\bf 1} (1994) \#R12.

\bibitem{AK1997}
N. Alon and M. Krivelevich,
{\it Constructive Bounds for a Ramsey-Type Problem},
Graphs Combin. {\bf 13}(3) (1997) 217--225.

\bibitem{BIP2020}
A. Bishnoi, F. Ihringer, and V. Pepe,
{\it A construction of clique-free pseudorandom graphs},
Combinatorica {\bf 40}(3) (2020) 307--314.

\bibitem{BK2010}
T. Bohman and P. Keevash,
{\it The early evolution of the H-free process},
Invent. Math. {\bf 181} (2010) 291--336.

\bibitem{Bollobas2001} 
B. Bollobás,
{\it Random Graphs}, 
2nd edition, Cambridge University Press (2001).

\bibitem{BT1981}
B. Bollob\'{a}s and A. Thomason,
{\it Graphs which Contain all Small Graphs},
Europ. J. Combinatorics {\bf 2} (1981) 13--15.

\bibitem{BCN}
A. E. Brouwer, A. M. Cohen, and A. Neumaier,
{\it Distance-regular graphs}, Springer, Heidelberg, 1989.

\bibitem{BvM}
A. E. Brouwer and H. Van Maldeghem,
{\it Strongly Regular Graphs},
Cambridge University Press,
Encyclopedia of Mathematics and Its Applications 182, 2022.

\bibitem{BO2015}
D. Bryant and P. \'{O} Cath\'{a}in,
{\it An asymptotic existence result on 
compressed sensing matrices},
Linear Algebra Appl. {\bf 475} (2015) 134--150.


\bibitem{DFG2009}
C. Dalf\'{o}, M.A. Fiol, and E. Garriga,
{\it On $k$-Walk-Regular Graphs},
Electronic J. Combin. {\bf 16} (2009) \#R47.

\bibitem{DGS1977}
Ph. Delsarte, J. M. Goethals, and J. J. Seidel,
{\it Spherical codes and designs},
Geom. Dedicata {\bf 6} (1977) 363--388.

\bibitem{Edel2004}
Y. Edel,
{\it Extensions of generalized product caps},
Des. Codes Cryptogr. {\bf 31}(1) (2004) 5--14.

\bibitem{EB1999}
Y. Edel and J. Bierbrauer,
{\it Recursive constructions for large caps},
Bull. Belg. Math. Soc. {\bf 6} (1999) 249--258.

\bibitem{EB2001}
Y. Edel and J. Bierbrauer,
{\it Large caps in small spaces},
Des. Codes Cryptogr. {\bf 23} (2001) 197--212.

\bibitem{EG2017}
J. S. Ellenberg and D. Gijswijt,
{\it On large subsets of $\Ff_q^n$ with no three-term progression},
Ann. of Math. {\bf 185} (2017) 339--343.

\bibitem{GJS2021}
S. Glock, O. Janzer, and B Sudakov,
{\it New results for MaxCut in $H$-free graphs},
arXiv:2104.06971v1 (2021).

\bibitem{Goldberg2006}
F. Goldberg, 
{\it On quasi-strongly regular graphs}, 
Linear and Multilinear Algebra {\bf 54}(6) (2006) 437--451.

\bibitem{GS2021}
S. Goryainov and L. V. Shalaginov,
{\it Deza graphs: a survey and new results},
arXiv:2103.00228v2 (2021).

\bibitem{GS2022}
G.~Greaves and J.~Syatriadi,
{\it Real equiangular lines in dimension 18 and the De Caen-Jacobi identity for complementary subgraphs},
arXiv:2206.04267v1 [math.CO] (2022).

\bibitem{GSS2022}
G.~Greaves, J.~Syatriadi, and P.~Yatsyna,
{\it Equiangular lines in Euclidean spaces: dimensions 17 and 18},
 arXiv:2104.04330v1 [math.CO] (2022).


\bibitem{Haemers1995}
W. H. Haemers, 
{\it Interlacing eigenvalues and graphs},
Linear Algebra Appl. {\bf 226--228} (1995) 593--616.

\bibitem{HS2001}
W. P. Hirschfeld and L. Storme, 
{\it The packing problem in statistics, coding theory
and finite projective spaces: Update 2001},
in Finite Geometries,
A. Blokhuis, J. W. P. Hirschfeld, D. Jungnickel,
and J. A. Thas, eds., Kluwer, Dodrecht, 2001, 201--246.

\bibitem{KeevashDesigns}
P. Keevash,
{\it The existence of designs},
 arXiv:1401.3665 [math.CO] (2014).

\bibitem{KS2006}
M. Krivelevich and B. Sudakov,
{\it Pseudo-random graphs},
in: Bolyai Soc. Math. Stud., vol. 15, Springer, Berlin, 199--262, 2006.

\bibitem{Meringer1999}
M. Meringer,
{\it Fast Generation of Regular Graphs and Construction of Cages},
J. Graph Theory {\bf 30} (1999) 137--146.

\bibitem{MV2019}
D. Mubayi and J. Verstra\"ete,
{\it A note on pseudorandom Ramsey graphs},
arXiv:1909.01461 (2019).

\bibitem{RX2007}
S. Radziszowski and X. Xiaodong,
{\it On the most wanted Folkman graph},
Geombinatorics XVI (2007) 367--381.

\bibitem{SDPJ2016}
G. Shi, D. Dong, I. R. Petersen, and
K. H. Johansson,
{\it Reaching a Quantum Consensus: Master Equations
That Generate Symmetrization and Synchronization},
IEEE Trans. Automat. Control {\bf 61}(2) (2016) 374--387.

\bibitem{Taylor91}
D. E. Taylor,
{\it The Geometry of the Classical Groups},
Heldermann, Berlin, 1992.

\end{thebibliography}
\end{document}